\documentclass[12pt]{article}
\usepackage{a4}
\usepackage{amsthm}
\usepackage{amsfonts}
\usepackage{amssymb}
\usepackage{amsmath}
\usepackage{enumitem}
\usepackage{xcolor}
\usepackage{cite}
\usepackage{epsfig}
\usepackage{hyperref}
\usepackage[capitalise]{cleveref}

\usepackage{tikz}
\tikzstyle{aNode} = [circle, fill = black]
\tikzstyle{bNode} = [circle,draw = black, thick]

\newcommand{\ppoints}[1]{%
\begin{tikzpicture}[inner sep = 0.7pt, #1]%
\node (1) at (0,-2) [aNode]{};
\node (3) at (1.5,-2) [aNode]{};
\node (2) at (0.75,-1) [aNode]{};
\end{tikzpicture}%
}

\def\points{\ppoints{scale=0.08}}

\newcommand{\ppedge}[1]{%
\begin{tikzpicture}[inner sep = 0.7pt, #1]%
\node (1) at (0,-2) [aNode]{};
\node (3) at (1.5,-2) [aNode]{};
\node (2) at (0.75,-1) [aNode]{};
\draw[thick] (0, -2) -- (1.5, -2);
\end{tikzpicture}%
}

\def\pedge{\ppedge{scale=0.08}}

\newcommand{\ppcherry}[1]{%
\begin{tikzpicture}[inner sep = 0.7pt, #1]%
\node (1) at (0,-2) [aNode]{};
\node (3) at (1.5,-2) [aNode]{};
\node (2) at (0.75,-1) [aNode]{};
\draw[thick] (0, -2) -- (0.75, -1);
\draw[thick] (1.5, -2) -- (0.75, -1);
\end{tikzpicture}%
}

\def\pcherry{\ppcherry{scale=0.08}}

\newtheorem{theorem}{Theorem}[section]

\newtheorem{lemma}[theorem]{Lemma}
\newtheorem{proposition}[theorem]{Proposition}
\newtheorem*{definition}{Definition}

\newtheorem{claim}[theorem]{Claim}
\newtheorem{question}[theorem]{Question}

\newcommand\NN{{\mathbb N}}

\newcommand\pa{{\pi_{\points}^{\operatorname{pal}}}}
\newcommand\pu{{\pi_{\points}}}
\newcommand\pve{{\pi_{\pedge}}}
\newcommand\pch{{\pi_{\pcherry}}}
\newcommand\cp{{\mathcal{P}}}
\newcommand\pdpal{{\pi_{\pedge}^{\operatorname{pal}}}}
\newcommand\pdch{{\pi_{\pcherry}^{\operatorname{pal}}}}

\DeclareTextCompositeCommand{\v}{OT1}{l}{l\nobreak\hspace{-.1em}'}
\DeclareTextCompositeCommand{\v}{OT1}{t}{t\nobreak\hspace{-.1em}'\nobreak\hspace{-.15em}}

\begin{document}
\title{Palettes determine uniform Tur\'an density}
\author{Ander Lamaison\thanks{Extremal Combinatorics and Probability Group (ECOPRO),  Institute for Basic Science (IBS), Daejeon, South Korea. Supported by IBS-R029-C4. \textbf{Email address}: ander@ibs.re.kr}}
\date{} 
\maketitle

\begin{abstract}
Tur\'an problems, which concern the minimum density threshold required for the existence of a particular substructure, are among the most fundamental problems in extremal combinatorics. We study Tur\'an problems for hypergraphs with an additional uniformity condition on the edge distribution. This kind of Tur\'an problems was introduced by Erd\H{o}s and S\'os in the 1980s but it took more than 30 years until the first non-trivial exact results were obtained when Glebov, Kr\'al' and Volec [Israel J. Math. 211 (2016), 349--366] and Reiher, R\"odl and Schacht [J. Eur. Math. Soc. 20 (2018), 1139--1159] determined the uniform Tur\'an density of $K_4^{(3)-}$.

Subsequent results exploited the powerful \emph{hypergraph regularity method}, developed by Gowers and by Nagle, R\"odl and Schacht about two decades ago. Central to the study of the uniform Tur\'an density of hypergraphs are \emph{palette constructions}, which were implicitly introduced by R\"odl in the 1980s. We prove that palette constructions always yield tight lower bounds, unconditionally confirming present empirical evidence. This results in new and simpler approaches to determining uniform Tur\'an densities, which completely bypass the use of the hypergraph regularity method.
\end{abstract}

\section{Introduction}

In extremal combinatorics, Tur\'an problems, 
which vastly generalize the classical Tur\'an's theorem from 1941,
concern the threshold density for the existence of 
a specific substructure in a host structure;
this threshold density is referred to as the \emph{Tur\'an density}.
While Tur\'an densities are very well-understood
in the case of graphs~\cite{Man07, Tur41, ErdS46},
Tur\'an problems concerning hypergraphs are
one of the most challenging problems in extremal combinatorics.
Indeed, Erd\H{o}s offered \$500 for determining the Tur\'an density
of \emph{any} complete hypergraph
and \$1000 for determining the Tur\'an densities of all complete hypergraphs.
Most of the extremal constructions for Tur\'an problems
in the hypergraph setting have large independent sets,
i.e., linear-sized sets of vertices with no edges.
This led Erd\H{o}s and S\'os~\cite{ErdS82, Erd90}
to propose studying the \emph{uniform Tur\'an density} of hypergraphs,
which is the density threshold for the existence of a hypergraph
with the additional requirement that
the edges of the host hypergraph are distributed uniformly.

Only recently, the resistance of uniform Tur\'an densities
has been broken using approaches
based on the hypergraph regularity method~\cite{ReiRS18a, BucCKMM23, GarKL24},
starting with resolving an almost 40-year-old problem
by Erd\H{o}s and S\'os on determining
the uniform Tur\'an density of the $3$-uniform hypergraph $K_4^{(3)-}$.
All known exact results match lower bounds
obtained by palette constructions~\cite{Rei20},
which extend the lower bound constructions
due to R\"odl~\cite{Rod86} from the 1980's.
Our main result asserts that this is a general phenomenon:
\emph{the uniform Tur\'an density of
every hypergraph $H$ is equal to the supremum
of the densities of palette constructions avoiding $H$}.
Since palette constructions are much simpler to analyze
than the reduced hypergraphs appearing in the approaches
based on the hypergraph regularity method
(as we demonstrate on the cases of
$K_4^{(3)-}$ and $C_\ell^{(3)}$ further),
our result brings a new tool for determining the
uniform Tur\'an density of hypergraphs,
in addition to resolving whether the uniform Tur\'an densities
are determined by palette constructions,
a problem which has been widely circulating around in the community
as discussed in~\cite[Section 3]{Rei20}.
As evidence to support this claim,
we show that there exists a $3$-uniform hypergraph
with uniform Tur\'an density equal to $\frac12-\frac{1}{2k}$
for all $k\geq 2$ (see Subsection~\ref{sec:palettes});
note that the set of known uniform Tur\'an densities
was finite prior to this work, specifically it consisted
of $0$, $1/27$, $4/27$, $1/4$ and $8/27$.

%A \emph{$r$-uniform hypergraph} $H$, or \emph{$r$-graph} for short,
%is a pair $(V,E)$, where $V$ is a set 
%whose elements are called \emph{vertices}
%and $E$ is a set of $r$-tuples of vertices, called the \emph{edges}.

We now define the concepts studied in the paper formally and
present them in a broader context.
For an $r$-uniform hypergraph $F$
(or $r$-graph for short),
 the \emph{extremal number} 
$\operatorname{ex}(n,F)$
is the maximum number of edges
of an $r$-graph $H$ with $n$ vertices
not containing $F$ as a subgraph.
The \emph{Tur\'an density} of an $r$-graph $F$
is defined as the limit
\[\pi(F):=\lim\limits_{n\rightarrow\infty}\frac{\operatorname{ex}(n,F)}{\binom{n}{r}}.\]
The existence of this limit follows
from the fact that the function
on the right hand side
is non-increasing on $n$
(see~\cite{KatNS64}).
The Tur\'an density of graphs is well-understood:
Tur\'an~\cite{Tur41}
determined the extremal numbers 
$\operatorname{ex}(n, K_k)$ of all complete graphs $K_k$,
while Erd\H{o}s and Stone~\cite{ErdS46}
proved that the Tur\'an density $\pi(G)$ of all graphs $G$
equals $\frac{\chi(G)-2}{\chi(G)-1}$,
where $\chi(G)$ is the chromatic number of $G$.

In contrast, for $r$-graphs $F$ with $r\geq 3$,
computing the value of $\pi(F)$
remains an elusive problem
even in some of the smallest cases.
To this day, the value of the Tur\'an density
$\pi(K_t^{(r)})$ of the complete $r$-graph with $t$ vertices
has not been obtained for any $t>r>2$.
Even for the very simple hypergraph $K_4^{(3)-}$,
obtained by removing an edge from $K_4^{(3)}$,
the Tur\'an density is unknown~\cite{Kee11}.
The best bounds on these values are
$5/9\leq \pi(K_4^{(3)})\leq 0.5616$
and $2/7\leq\pi(K_4^{(3)-})\leq 0.2871$,
where the upper bounds were obtained
using Razborov's flag algebra method~\cite{BabT11, Raz10}.

It is worth noting that for many hypergraphs $F$,
the conjectured extremal constructions
of $F$-free hypergraphs $H$
have edges very unevenly distributed.
For example, in the case $F=K_4^{(3)}$,
in the original $F$-free construction with edge density $5/9$
due to Tur\'an~\cite{Tur61},
the vertex set can be split 
into three independent sets of size $n/3$.
This motivated Erd\H{o}s and S\'os~\cite{ErdS82}
to introduce Tur\'an problems
with an additional edge distribution condition.

\begin{definition}
A $3$-graph $H$ is said to be \emph{$(d,\varepsilon,\points)$-dense}
if any subset $S\subseteq V(H)$ contains at least $d\binom{|S|}{3}-\varepsilon|V(H)|^3$ edges. 
The \emph{uniform Tur\'an density} $\pu(F)$ of a $3$-graph $F$
is defined as the infimum of the values of $d$, 
for which there exists $\varepsilon>0$ and $N$ such that 
every $(d,\varepsilon,\points)$-dense hypergraph
on at least $N$ vertices
contains $F$ as a subgraph.
\end{definition}

In~\cite{Erd90}, Erd\H{o}s conjectured
that $\pu(K_4^{(3)})=1/2$ and $\pu(K_4^{(3)-})=1/4$. 
While the former remains open,
Glebov, Kr\'al' and Volec~\cite{GleKV16}
gave a computer-assisted proof of the latter conjecture,
which was then proved combinatorially by
Reiher, R\"odl and Schacht~\cite{ReiRS18a}.

We now briefly survey recent results
on exact values of the uniform Tur\'an densities of $3$-graphs.
Reiher, R\"odl and Schacht~\cite{ReiRS18}
characterized $3$-graphs $F$ with $\pu(F)=0$.
As a consequence of this characterization,
they deduced that
every $3$-graph $F$ with non-zero uniform Tur\'an density
satisfies $\pu(F)\geq 1/27$.
In other words, there is a ``jump'' phenomenon:
$\pu(F)$ does not take values in $(0, 1/27)$.
Garbe, Kr\'al' and the author~\cite{GarKL24}
constructed $3$-graphs with uniform Tur\'an density $1/27$.
Other classes of $3$-graphs whose uniform Tur\'an density
are known are tight cycles~\cite{BucCKMM23},
and a specific family of $3$-graphs
with uniform Tur\'an density $8/27$~\cite{GIKKL24}.
In all these cases, palette constructions,
which we introduce next, play a key role.

\subsection{Palettes}\label{sec:palettes}

The tight lower bounds for all known values
of the uniform Tur\'an densities of $3$-graphs
arise from palette constructions.
This concept was introduced by Reiher~\cite{Rei20}, 
extending a construction by R\"odl~\cite{Rod86}.

\begin{definition}\label{def:palette}
A \emph{palette} $\cp$ is a pair $(\mathcal{C},\mathcal{A})$, 
where $\mathcal{C}$ is a finite set (whose elements we call \emph{colors}) 
and a set of (ordered) triples of colors $\mathcal{A}\subseteq \mathcal{C}^3$, 
which we call the \emph{admissible triples}.
The density of $\cp$ is $d(\cp):=|\mathcal{A}|/|\mathcal{C}|^3$.

We say that a $3$-graph $F$ \emph{admits} a palette $\cp$
if there exists an order $\preceq$ on $V(F)$
and a function $\varphi:\binom{V(F)}{2}\rightarrow \mathcal{C}$
such that for every edge $uvw\in E(F)$ with $u\prec v\prec w$
we have $\left(\varphi(uv), \varphi(uw), \varphi(vw)\right)\in \mathcal{A}$.
\end{definition}

Palettes can be used to obtain lower bounds on uniform Tur\'an densities.
Specifically, if $F$ does not admit 
a palette $\cp$, then $\pu(F)\geq d(\cp)$.
The reason is that $\cp$ can be used
to generate a $(d(\cp), o(1), \points)$-dense
$F$-free $3$-graph $H_n$ with $n$ vertices.
To construct the hypergraph $H_n$,
proceed as follows:
the vertex set of $H_n$ is $[n]$.
Randomly color the edges of the complete ($2$-)graph $K_n$
with the colors from $\mathcal{C}$.
Now a triple of vertices $u<v<w$
is an edge in $H_n$
if the ordered triple of colors assigned to $uv$, $uw$ and $vw$
(in this order) belongs to $\mathcal{A}$.
In addition to lower bounds, in some cases 
palettes can be used to give characterizations of upper bounds~\cite{GarKL24, LiLWZ23, GIKKL24}.
Most notably, Reiher, R\"{o}dl and Schacht~\cite{ReiRS18}
proved that a $3$-graph $F$ has $\pu(F)=0$ if and only if
$F$ admits the three-color palette with 
$\mathcal{A}=\{(\text{red, green, blue})\}$.

All the lower bound constructions
for the tight results on uniform Tur\'an density mentioned above
are derived from palettes via this procedure.
The same applies to the conjectured optimal constructions
for other families of $3$-graphs, 
including complete graphs and stars~\cite{Rei20}.
Our main result asserts that this is a general phenomenon.
In particular, we show that $\pu(F)=\pa(F)$
for every $3$-graph $F$,
where $\pa$ is defined as follows.

\begin{definition}
 The \emph{palette Tur\'an density} of a 3-graph $F$ is
 
 \[\pa(F):=\sup\{d(\cp):\cp\text{ palette, }F\text{ does not admit }\cp\}.\]
%denoted by $\pa(F)$,
%is the supremum of the values of $d$
%for which there exists a palette $\cp$
%with density at least $d$
%such that $F$ does not admit $\cp$.
\end{definition}

%Reiher asked whether all lower bounds for $\pu$
%can be derived from palettes.
%We answer this question in the affirmative.

\begin{theorem}\label{thm:palette}
For every $3$-graph $F$, we have $\pu(F)=\pa(F)$.
\end{theorem}

Theorem~\ref{thm:palette} answers a question,
which was circulating in the community and is
explicitly discussed in the survey 
by Reiher~\cite[Section 3]{Rei20}.
An important impact of Theorem~\ref{thm:palette}
is that computing $\pa(F)$
is generally easier than computing $\pu(F)$.
As one instance of this,
we will give a short proof of $\pu(K_4^{(3)-})=1/4$ in Section~\ref{sec:applications}.
Another example showcasing this phenomenon is the cycle $C_\ell^{(3)}$.
In his Master's thesis, Cooper~\cite{Coo18}
proved in 2018 that $\pa(C_5^{(3)})=4/27$.
Using blow-ups, one can easily derive that
$\pa(C_\ell^{(3)})=4/27$ holds for every $\ell\geq 8$ not divisible by 3.
It took three years and a considerable amount of effort
to prove that $\pu(C_\ell^{(3)})=4/27$.
A side-by-side comparison of the proofs
reveals the additional complexity
that the study of $\pu$ presents with respect to $\pa$.
Therefore, Theorem~\ref{thm:palette}
can act as a ``black box'' to translate (generally simpler)
palette proofs into the setting of uniform Tur\'an density.

We now discuss specific new applications of Theorem~\ref{thm:palette}.
The power of Theorem~\ref{thm:palette}
is shown in the follow-up paper~\cite{LamW24},
where Wu and the author use the theorem
to determine the uniform Tur\'an density of large stars.
As we mentioned, $0$, $1/27$, $4/27$, $1/4$
and $8/27$ are in fact
the only known values of $\pu$.
In this paper, we use Theorem~\ref{thm:palette}
to find an infinite sequence of values of $\pu$:

\begin{theorem}\label{thm:inffamily}
For every $k\geq 2$ there exists a $3$-graph $F$ with $\pu(F)=\frac12-\frac{1}{2k}$.
\end{theorem}

Theorem~\ref{thm:inffamily} implies in particular that $1/2$
is an accumulation point for the values of $\pu$.
We remark that a recent result of Conlon and Sch\"ulke~\cite{ConS24}
shows that $1/2$ is an accumulation point
in the setting of Tur\'an density,
i.e., for the values of $\pi$.
Unlike in~\cite{ConS24}, we compute the sequence
of values of $\pu$ that converges to $1/2$ explicitly.

The parameter $\pu$ is not the only variant of Tur\'an density for which
Reiher suggested a connection to palettes in~\cite[Section 3]{Rei20}.
There are two additional variants,
introduced by Reiher, R\"odl and Schacht in \cite{ReiRS16, ReiRS18b}
and denoted by $\pve$ and $\pch$.
In Section~\ref{sec:pve} we will introduce $\pve$ and
we will prove an analogue of Theorem~\ref{thm:palette} for this parameter,
and in Section~\ref{sec:pve0} we will use it
to characterize $3$-graphs $F$ with $\pve(F)=0$.
We remark that such a characterization 
was also announced by Reiher, R\"{o}dl and Schacht.
In Section~\ref{sec:remarks} we will discuss
possible extensions to $\pch$.

\subsection{Reduced hypergraphs}

The most important tool
in the proof of Theorem~\ref{thm:palette}
is the concept of reduced hypergraph,
which was formally introduced in~\cite{Rei20}
and which we now present.

\begin{definition}
Let $N$ be a positive integer. 
An \emph{$N$-reduced hypergraph} is composed of a 3-graph $H$
together with a set of $N$ indices $U$.
$V(H)$ is the disjoint union of
 ${N \choose 2}$ vertex sets $V_{\alpha,\beta}$,
with $\{\alpha,\beta\}\in {U \choose 2}$,
and the edge set of $H$ is the union of 
${N \choose 3}$ tripartite graphs $A_{\alpha,\beta,\gamma}$
on $V_{\alpha, \beta},V_{\alpha, \gamma}, V_{\beta, \gamma}$
for $\{\alpha, \beta, \gamma\}\in {U \choose 3}$,
called its \emph{constituents}.
Note that we will treat the subindices
as unordered sets,
so $V_{\alpha, \beta}=V_{\beta, \alpha}$.
We say that $H$ has density at least $d$
if for all triples $\{\alpha,\beta,\gamma\}\in {U \choose 3}$
the constituent $A_{\alpha,\beta,\gamma}$
has at least $d|V_{\alpha, \beta}||V_{\alpha, \gamma}||V_{\beta, \gamma}|$ edges.
\end{definition}

It is useful (and not far from the reality)
to think of $N$-reduced hypergraphs
as a result of applying the hypergraph regularity lemma
to a large host $3$-graph.
We now cast a way that a graph $F$
is deduced to exist in the host hypergraph
from its regularity partition
in the setting of reduced hypergraphs.

\begin{definition}
Given a 3-graph $F$ and an $N$-reduced hypergraph $H$,
we say that $H$ \emph{embeds} $F$
if there exists an injective function $\tau:V(F)\rightarrow U$
and a function $\varphi:{V(F)\choose 2}\rightarrow V(H)$
such that for every $v,w\in V(F)$
we have $\varphi(uv)\in V_{\tau(u), \tau(v)}$,
and for every $uvw\in E(H)$
we have that $\varphi(uv)\varphi(uw)\varphi(vw)\in E(H)$.
\end{definition}

The regularity lemma produces
an equipartition of the vertex set $V(H')$
of a host hypergraph $H'$ into a bounded number of classes,
and a partition of ${V(H') \choose 2}$ as well.
Here, the set of indices $U$
roughly corresponds to the family of vertex classes,
and each vertex in $H$ corresponds
to a subset of ${V(H') \choose 2}$.
If $H'$ is $(d, o(1), \points)$-dense,
then after performing some clean-up,
the resulting reduced hypergraph $H$
will have density
at least $d-o(1)$.

This relation between reduced hypergraphs and the regularity lemma
is behind the following result of Reiher,
which characterizes $\pu$
in terms of reduced hypergraphs.

\begin{proposition}[{\cite[Theorem 3.3]{Rei20}}]
\label{prop:reiher}
Let $F$ be a $3$-graph. Then $\pu(F)$ is the supremum of the values of $d$ such that, for every $N$,
there exists an $N$-reduced hypergraph with density at least $d$ which does not embed $F$.
\end{proposition}

\subsection{Structure of the paper}

In Section~\ref{sec:structure} we will present 
an overview of the proof of Theorem~\ref{thm:palette}. 
We will introduce the probabilistic and Ramsey-theoretic 
tools needed for our proof in Section~\ref{sec:prelim}. 
Then Section~\ref{sec:mainproof} is devoted 
to proving Theorem~\ref{thm:palette}. 
In Section~\ref{sec:pve} we will define the parameter $\pve$,
and we will prove an analogous result to Theorem~\ref{thm:palette},
which we present as Theorem~\ref{thm:pve}.
To demonstrate the range of applications of our main theorems, 
in Section~\ref{sec:applications} we will 
give a short proof that $\pu(K_4^{(3)-})=1/4$, we derive 
Theorem~\ref{thm:inffamily} on the existence of
hypergraphs with uniform Tur\'an density
equal to $\frac12-\frac{1}{2k}$, 
and we characterize 
the $3$-graphs $F$ with $\pve(F)=0$. 
In Section~\ref{sec:remarks} 
we will explain the difficulty of applying 
the same method to $\pch$, 
and finally in Section~\ref{sec:finremarks} 
we discuss in which cases does there exist
a $3$-graph $F$ admitting some palettes
$\cp_1, \cp_2, \dots, \cp_m$
while not admitting other palettes
$\mathcal{Q}_1,\mathcal{Q}_2, \dots, \mathcal{Q}_n$.

\section{Structure of the proof}\label{sec:structure}

At its core, the proof strategy 
for Theorem~\ref{thm:palette} goes as follows.
We start with a 3-graph $F$
with $\pu(F)=\pi$ and $\varepsilon>0$.
Our goal is to construct a palette $\cp$
with density at least $\pi-\varepsilon$
not admitted by $F$.

We note that, given a palette $\cp=(\mathcal{C}, \mathcal{A})$ 
with density $d$ and a natural number $n$,
we can construct an $n$-reduced hypergraph $\cp[n]$ with density $d$,
by taking each set $V_{\alpha,\beta}$ to be a copy of $\mathcal{C}$,
and for every triple of indices $1\leq\alpha<\beta<\gamma\leq n$
letting the constituent $A_{\alpha, \beta, \gamma}$
be a copy of $\mathcal{A}$,
where for every triple $(c_1, c_2, c_3)\in\mathcal{A}$
we take an edge with $c_1$ in $V_{\alpha, \beta}$, 
$c_2$ in $V_{\alpha, \gamma}$ and $c_3$ in $V_{\beta, \gamma}$.
Moreover, when $n=|V(F)|$, we have that
$\cp[n]$ embeds $F$ if and only if $F$ admits the palette $\cp$
(see Lemma~\ref{lem:embedding}).

Given $F$ and $\varepsilon$, we can use Proposition~\ref{prop:reiher}
to find an $N$-reduced hypergraph $H$
with density at least $\pi-\varepsilon/2$,
for $N$ arbitrarily large, which does not embed $F$.
We will find a subgraph in $H$ of the form $\cp[n]$,
for some palette $\cp$ of density at least $\pi-\varepsilon$.
Since $H$ does not embed $F$, neither does $\cp[n]$,
and so $F$ does not admit $\cp$.

The key step in the proof is the following lemma:

\begin{lemma}\label{lem:lowres}
For all $\varepsilon>0$ there exist $s=s(\varepsilon)$,
such that for all $m$ there exists $N=N(m,\varepsilon)$ 
for which the following holds: 
if $H$ is an $N$-reduced hypergraph on index set $[N]$ 
with density at least $d$, 
there exists a subset $U\subseteq [N]$ of $m$ indices, 
and for each $\alpha,\beta\in U$ there exists 
a multiset $S_{\alpha,\beta}$ of $s$ vertices in $V_{\alpha,\beta}$,
such that the $m$-reduced hypergraph induced by $H$
on the sets $S_{\alpha,\beta}$ has density at least $d-\varepsilon$.

\end{lemma}

The intuition behind this lemma is as follows. 
In the $N$-reduced hypergraph $H$, 
each vertex set $V_{\alpha,\beta}$ could be arbitrarily large, 
and as such the constituents $A_{\alpha,\beta,\gamma}$
could be arbitrarily complex. 
We want to decrease this complexity by reducing each set $V_{\alpha,\beta}$ 
into a subset with a bounded number of vertices, 
at the cost of decreasing the number of indices
from $N$ to $m$
and decreasing the density of $H$ by $\varepsilon$. 
If $s$ is allowed to depend on $\varepsilon$ and $m$, 
then it is not too hard to prove that a random choice of multisets $S_{\alpha,\beta}$
will succeed with high probability for large enough $s$.
The crucial point of Lemma~\ref{lem:lowres}
is that $s$ does not depend on $m$, only on $\varepsilon$.
Lemma~\ref{lem:lowres} is proved through a suitable 
combination of random vertex selections and applications of Ramsey's theorem.
The probabilistic and Ramsey-theoretic tools necessary
will be introduced in Section~\ref{sec:prelim}.

After applying Lemma~\ref{lem:lowres} and obtaining 
an $m$-reduced hypergraph $H'$ of density at least $\pi-\varepsilon$
where each part $S_{\alpha,\beta}$ has exactly $s$ vertices,
identify the elements of each vertex set with $[s]$ arbitrarily.
The resulting constituents $A'_{\alpha, \beta, \gamma}$ on $S_{\alpha,\beta}\times S_{\alpha,\gamma}\times S_{\beta,\gamma}$
can only be one of a bounded number of possibilities.
Applying Ramsey's theorem we find a subset of $n$ indices
in which all constituents look the same.
This produces a 3-graph of the form $\cp[n]$, as we wanted to find.

\section{Preliminaries}\label{sec:prelim}

The hardest part of the proof of Theorem~\ref{thm:palette}
is proving Lemma~\ref{lem:lowres}.
In short, we want to take the $N$-reduced hypergraph $H$
and obtain a ``low-resolution'' subgraph preserving most of the density.
The next lemma will be useful in taking such discretizations
while preserving the average value of a certain function,
which in our case will relate to the degree of vertices in certain subgraphs.

We will consider functions of the form $\mu:S\rightarrow [0,1]$.
Given a (multi)-subset $X$ of $S$, 
we will denote the average value of $\mu$ on $X$ by $\bar\mu(X)$.
The sum of two multisets $X_1+X_2$ is the multiset $X$ 
in which the multiplicity of each element $x$ 
is the sum of its multiplicities in $X_1$ and $X_2$.
We remark that $\bar\mu(X_1+X_2)\geq\min\{\bar\mu(X_1), \bar\mu(X_2)\}$.

Our main probabilistic tool will be Hoeffding's inequality. 
This inequality tells us that, if we sample a vector $X$ from $S$ uniformly at random,
which we will treat as a multiset,
the value of $\bar\mu(X)$ is very highly concentrated around $\bar\mu(S)$.

\begin{lemma}[Hoeffding's inequality]\label{lem:hoeffding}
Let $\mu:S\rightarrow [0,1]$ be a function,
 let $t$ be a positive integer and let $\varepsilon>0$. 
 Suppose that a vector $X=(x_1, x_2, \dots, x_t)$ 
 is sampled uniformly at random from $S^t$. Then

\[\Pr\left(\bar\mu(X)<\bar\mu(S)-\varepsilon\right)\leq e^{-2\varepsilon^2t}.\]

\end{lemma}

We will also use Ramsey's theorem several times.
Ramsey's theorem states that,
for all $k,n,r$ there exists $N=R_r(n,k)$
such that whenever the edges of the
complete $r$-graph $K_N^{(r)}$ on $N$ vertices
are colored in $k$ colors,
there exists a subset of $n$ vertices
in which all edges have the same color.

The second lemma that we will use in the proof of Lemma~\ref{lem:lowres}
has to do with the way in which the vertices of $S_{\alpha.\beta}$
will be selected in the proof of the lemma.
Suppose that for each pair of indices $\alpha,\beta$ 
we want to select a vertex $v_{\alpha,\beta}\in V_{\alpha,\beta}$
while avoiding certain unlikely ``bad'' events within the constituents.
Specifically, for each index $\gamma$ different from $\alpha$ and $\beta$
we introduce a restriction for the choice of $v_{\alpha,\beta}$.
The lemma then says that there exists a subset $U$ of the indices
in which all bad events can be avoided simultaneously.

\begin{lemma}\label{lem:selection}
For every $m$ there exists $n$ with the following property.
Let $H$ be an $n$-reduced hypergraph.
Suppose that for each triple of distinct indices $\alpha,\beta,\gamma$
we have a set $B_{\alpha,\beta}^\gamma\subset V_{\alpha, \beta}$
with $|B_{\alpha,\beta}^\gamma|\leq 0.1 |V_{\alpha, \beta}|$.
Then there exists a subset of $m$ indices $U$
and for each pair $\alpha,\beta\in U$ there exists $v_{\alpha, \beta}\in V_{\alpha, \beta}$
such that for all $\alpha,\beta,\gamma\in U$
we have $v_{\alpha,\beta}\notin B_{\alpha,\beta}^\gamma$.
\end{lemma}

The constant 0.1 in Lemma~\ref{lem:selection} can be replaced 
by any number smaller than $1/3$ with a slightly more careful analysis, but it fails for $1/3$.

\begin{proof}[Proof of Lemma~\ref{lem:selection}]
Assume the opposite. For every $m$-tuple $u_1<u_2<\dots<u_m$ of indices, 
there exist $\alpha,\beta=u_i,u_j$ such that
\[\bigcup\limits_{k\in [m]\setminus\{i,j\}}B_{\alpha, \beta}^{u_k}=V_{\alpha, \beta}.\]
Otherwise, we could select $v_{\alpha, \beta}$ outside this union
for all pairs $\alpha, \beta$.
This means that to each $m$-tuple of indices
we can assign a pair $\{i,j\}\in {[m] \choose 2}$.
If $n$ is the Ramsey number $R_m\left(6m+2,{m \choose 2}\right)$,
we can find a $6m+2$-tuple of indices
where all $m$-tuples receive the same pair $\{i,j\}$.
We can consider that this $6m+2$-tuple is $[6m+2]$.

Now, fix $\alpha=2m+1$ and $\beta=4m+2$. 
For each $\gamma\in [6m+2]\setminus\{\alpha,\beta\}$, 
we have $|B_{\alpha,\beta}^\gamma|\leq 0.1 |V_{\alpha, \beta}|$. 
Therefore, there are at most $0.6 |V_{\alpha, \beta}|$ vertices in $V_{\alpha, \beta}$ 
which belong to at least $m$ of these sets $B_{\alpha, \beta}^\gamma$. 
Thus there exists $v\in V_{\alpha, \beta}$ which belongs to at most $m$ of these sets.
 In particular, there are at least $m$ indices $\gamma$ 
 in each of the intervals $[1,2m]$, $[2m+2, 4m+1]$ and $[4m+3, 6m+2]$ 
 such that $v\notin B_{\alpha, \beta}^\gamma$. 
 Using this, we can find an $m$-tuple $1\leq u_1<u_2<\dots<u_m\leq 6m+2$ 
 with $u_i=\alpha$, $u_j=\beta$, and $v\notin B_{\alpha, \beta}^{u_k}$ for all $k\in [m]\setminus\{i,j\}$.
  This contradicts the fact that the $m$-tuple receives the pair $\{i,j\}$.
  \end{proof}

To conclude these preliminaries,
we will prove the following lemma,
that will give us the final step of the proof of Theorem~\ref{thm:palette}.
Remember that, given a palette $\cp$,
we defined an $n$-reduced hypergraph $\cp[n]$
in Section~\ref{sec:structure}:

\begin{lemma}\label{lem:embedding}
Let $F$ be a $3$-graph on $n$ vertices, 
and let $\cp=(\mathcal{C}, \mathcal{A})$ be a palette. 
Then $F$ admits $\cp$ if and only if $\cp[n]$ embeds $F$.
\end{lemma}

\begin{proof}
Let $\cp=(\mathcal{C}, \mathcal{A})$. If $F$ admits $\mathcal{P}$, let $v_1\prec v_2\prec\dots\prec v_n$ be the order on $V(F)$ and $\varphi:{V(H) \choose 2}\rightarrow\mathcal{C}$ be the function that certify this fact. Then, to embed $F$ in $\cp[n]$, simply take the function $\tau(v_i)=i$ and $\varphi'$ sending $v_iv_j$ to the copy of $\varphi(v_iv_j)$ in $V_{i,j}$. This satisfies that, for all pairs of vertices $u,v\in V(F)$ we have $\varphi'(uv)\in V_{\tau(u)\tau(v)}$, and for every $uvw\in E(F)$ we have that $\varphi'(uv)\varphi'(uw)\varphi'(vw)\in E(H)$. 

On the other hand, if $\cp[n]$ embeds $F$, then $\tau:V(F)\rightarrow [n]$ must be bijective. Let $\preceq$ be the order that $\tau$ induces on $V(F)$, and color each pair of vertices with its image under $\varphi$. Then if $u\prec v\prec w$ form an edge of $F$, the colors of $uv$, $uw$ and $vw$ form an ordered triple of $\mathcal{A}$.
\end{proof}

\section{Proof of Theorem~\ref{thm:palette}}\label{sec:mainproof}

We will start by proving Lemma~\ref{lem:lowres}, 
and then use it to prove Theorem~\ref{thm:palette}. 
We have already described how to use the lemma to prove the theorem, 
so let us go into a bit more detail 
about the proof of Lemma~\ref{lem:lowres} itself. 
The number $s$ of vertices in each part $S_{\alpha, \beta}$ will be $s=rt$, 
where $r$ and $t$ each depend only on $\varepsilon$. 
We will follow an algorithm which consists of $r$ rounds, 
and on each round we will select $t$ vertices. 

The set of active indices after $i$ rounds will be $U_i$,
where $[N]=U_0\supseteq U_1\supseteq\dots\supseteq U_r$ and $|U_r|=m$.
For $\alpha, \beta\in U_i$, we will denote by $T_{\alpha, \beta}^{(i)}$ the multiset of $t$ vertices 
from $V_{\alpha,\beta}$ selected on the $i$-th round of the algorithm, 
and $S_{\alpha,\beta}=\sum_{j=1}^rT_{\alpha,\beta}^{(j)}$ will be the sum of these multisets 
(remember that multiset sum is defined so that the multiplicity of each element is additive).

For each triple of indices $\alpha, \beta, \gamma\in U_r$, 
we would like to control the number of edges in 
$S_{\alpha, \beta}\times S_{\alpha, \gamma}\times S_{\beta, \gamma}$.
A drawback of our method is that it is hard to say anything
about the codegree of vertices $v\in T_{\alpha,\beta}^{(i)}$ and $w\in T_{\alpha,\gamma}^{(i)}$
selected on the same round of the algorithm.
Fortunately, it will not be necessary to do so.
Instead, for any triple $1\leq i<j<k\leq r$ of rounds
and for all triples of indices $\alpha, \beta, \gamma\in U_k$
we will have that $T_{\alpha, \beta}^{(i)}\times T_{\alpha, \gamma}^{(j)}\times T_{\beta, \gamma}^{(k)}$
contains at least $(d-\varepsilon/2)t^3$ edges.
Adding over all triples $1\leq i<j<k\leq r$ and all permutations of $\alpha, \beta, \gamma$
we obtain that 
\[E(S_{\alpha, \beta}, S_{\alpha, \gamma}, S_{\beta, \gamma})\geq r(r-1)(r-2)(d-\varepsilon/2)t^3.\]
Then for $r=\lceil6\varepsilon^{-1}\rceil$, 
this number is at least $(d-\varepsilon)(rt)^3=(d-\varepsilon)s^3$.

\begin{proof}[Proof of Lemma~\ref{lem:lowres}]
Fix $r=\lceil 6\varepsilon^{-1}\rceil$. Fix $t=\lceil36\varepsilon^{-2}\log(10r^2)\rceil$
so that $e^{-(\varepsilon/6)^2t}\leq 1/(10r^2)$, and set $s=rt$.
Consider integers $N_0>N_1>\dots>N_r$,
where $N_r=m$ and $N_{k-1}=n(N_{k})$ as in the statement of Lemma~\ref{lem:selection}.
Finally, fix $N=N_0$, and $U_0=[N]$.

Let $H$ be an $N$-reduced hypergraph with density at least $d$.
For every pair $\alpha,\beta\in U_0$, let 
$(V_{\alpha,\beta})^t$ be the set of vectors
of length $t$ with entries in $V_{\alpha, \beta}$.
On each round of the algorithm, we will select
the set $T_{\alpha, \beta}^{(i)}\in (V_{\alpha, \beta})^t$
by applying Lemma~\ref{lem:selection} to the sets $(V_{\alpha, \beta})^t$,
for a specific choice of bad sets $B_{\alpha, \beta}^\gamma$.

The end goal in our algorithm is to ensure that, for all $\alpha,\beta,\gamma\in U_r$,
and every $1\leq i<j<k\leq r$,
the tripartite graph on 
$T_{\alpha, \beta}^{(i)}\times T_{\alpha, \gamma}^{(j)}\times T_{\beta, \gamma}^{(k)}$
contains at least $(d-\varepsilon/2)t^3$ edges.
We will keep certain invariants during the process
to guarantee that random choices in future rounds succeed with high enough probability.
Here, $\mu(A,B,C)$ indicates the edge-density of $H$ 
on the tripartite graph on $A\times B\times C$.
These invariants are:
\begin{itemize}
\item $\mu(T_{\alpha,\beta}^{(k)}, V_{\alpha, \gamma}, V_{\beta, \gamma})\geq d-\varepsilon/6$ for all $1\leq k\leq r$ and $\alpha, \beta, \gamma\in U_k$.
\item $\mu(T_{\alpha,\beta}^{(k)}, T_{\alpha, \gamma}^{(j)}, V_{\beta, \gamma})\geq d-\varepsilon/3$ for all $1\leq j<k\leq r$ and $\alpha, \beta, \gamma\in U_k$.
\item $\mu(T_{\alpha,\beta}^{(k)}, T_{\alpha, \gamma}^{(j)}, T_{\beta, \gamma}^{(i)})\geq d-\varepsilon/2$ for all $1\leq i<j<k\leq r$ and $\alpha, \beta, \gamma\in U_k$.
\end{itemize}

Let $1\leq k\leq r$, and let $\alpha, \beta, \gamma\in U_{k-1}$.
We will analyze the $k$-th round of the algorithm.
Assume that all invariants are preserved up to round $k-1$.
We want to make sure that, if all of these indices end up in $U_k$, 
then the set $T_{\alpha, \beta}^{(k)}$ does not break any of the invariants. 
We set $B_{\alpha, \beta}^{\gamma}\subseteq (V_{\alpha, \beta})^t$ 
to be the set of elements that break one or more of the invariants,
if selected as $T_{\alpha, \beta}^{(k)}$. 
We will show that the size of this set is at most $0.1|V_{\alpha, \beta}|^t$,
meaning that we will be in the setup of Lemma~\ref{lem:selection}.

For a fixed choice of $\alpha, \beta, \gamma$ we will consider several functions $\mu:V_{\alpha, \beta}\rightarrow[0,1]$. These will be, for all $1\leq i, j\leq k-1$ with $i\neq j$:
\begin{align*}\mu(v)=&\mu(\{v\}, V_{\alpha, \gamma}, V_{\beta, \gamma}),&\text{ which satisfies }& &\bar\mu(V_{\alpha, \beta})&=\mu(V_{\alpha, \beta}, V_{\alpha, \gamma}, V_{\beta, \gamma})\geq d.\\
\mu_j(v)=&\mu(\{v\}, T_{\alpha, \gamma}^{(j)}, V_{\beta, \gamma}),&\text{ which satisfies }& &\bar\mu_j(V_{\alpha, \beta})&=\mu(V_{\alpha, \beta}, T_{\alpha, \gamma}^{(j)}, V_{\beta, \gamma})\geq d-\varepsilon/6.\\
\mu'_j(v)=&\mu(\{v\}, V_{\alpha, \gamma}, T_{\beta, \gamma}^{(j)}),&\text{ which satisfies }& &\bar\mu'_j(V_{\alpha, \beta})&=\mu(V_{\alpha, \beta}, V_{\alpha, \gamma}, T_{\beta, \gamma}^{(j)})\geq d-\varepsilon/6.\\
\mu_{(j,i)}(v)=&\mu(\{v\}, T_{\alpha, \gamma}^{(j)}, T_{\beta, \gamma}^{(i)}),&\text{ which satisfies }& &\bar\mu_{(j,i)}(V_{\alpha, \beta})&=\mu(V_{\alpha, \beta}, T_{\alpha, \gamma}^{(j)}, T_{\beta, \gamma}^{(i)})\geq d-\varepsilon/3.
\end{align*}

We have one function $\mu$, $k-1$ functions of the form $\mu_j$, 
$k-1$ functions of the form $\mu'_j$ 
and $(k-1)(k-2)$ functions of the form $\mu_{(j,i)}$.
That yields a total of no more than $k^2\leq r^2$ functions considered. 
By Lemma~\ref{lem:hoeffding},
when $T_{\alpha, \beta}^{(k)}$ is selected
from $(V_{\alpha, \beta})^t$ uniformly at random,
for each of these functions $\mu$, we have 
\[\Pr\left(\bar\mu(T_{\alpha, \beta}^{(k)})\leq \bar\mu(V_{\alpha, \beta})-\varepsilon/6\right)\leq e^{-(\varepsilon/6)^2t}\leq \frac{1}{10r^2}.\]
Hence with probability at least $0.1$, 
all of the functions considered above have averages on $T_{\alpha, \beta}^{(k)}$ 
which are below the expected value by no more than $\varepsilon/6$, 
meaning that $|B_{\alpha, \beta}^\gamma|\leq 0.1|V_{\alpha, \beta}|^t$.
 
Apply Lemma~\ref{lem:hoeffding} to the sets 
$(V_{\alpha, \beta})^t$ and $B_{\alpha, \beta}^\gamma$.
This yields a subset of indices $U_k\subseteq U_{k-1}$
and a choice of $T_{\alpha, \beta}^{(k)}$ preserving all the invariants. 
Since $|U_{k-1}|=N_{k-1}=n(N_k)$, we have $|U_k|=N_k$. 
After $r$ rounds, we have $|U_r|=N_r=m$, 
and all throughout the algorithm we have that 
$\mu(T_{\alpha,\beta}^{(k)}, T_{\alpha, \gamma}^{(j)}, T_{\beta, \gamma}^{(i)})\geq d-\varepsilon/2$, 
or equivalently $E(T_{\alpha,\beta}^{(k)}, T_{\alpha, \gamma}^{(j)}, T_{\beta, \gamma}^{(i)})\geq (d-\varepsilon/2)t^3$. 
Taking the sets $S_{\alpha,\beta}=\sum_{i=1}^rT_{\alpha, \beta}^{(i)}$, 
we see that $E(S_{\alpha, \beta}, S_{\alpha, \gamma}, S_{\beta, \gamma})\geq (d-\varepsilon)s^3$, 
as explained earlier.\end{proof}

Now that we have Lemma~\ref{lem:lowres}, 
we can finish the proof of Theorem~\ref{thm:palette} 
by applying Ramsey's theorem to the hypergraph induced by $H$ on the sets $S_{\alpha, \beta}$, 
to ensure that all triples of indices induce the exact same constituent. 
The result is then a hypergraph of the form $\cp[n]$, 
for some palette $\cp$.

\begin{proof}[Proof of Theorem~\ref{thm:palette}]
Let $F$ be a 3-graph, let $\pi=\pu(F)$, and let $\varepsilon>0$.
Our goal is to obtain a palette $\cp$ with density at least $\pi-\varepsilon$
such that $F$ does not admit $\cp$.
Set $n=|V(F)|$, 
set $s=s(\varepsilon/2)$ from the statement of Lemma~\ref{lem:lowres},
set $m=R_3(n,2^{s^3})$,
and set $N=N(m,\varepsilon/2)$ from the statement of Lemma~\ref{lem:lowres}.
Applying Proposition~\ref{prop:reiher} with $\delta=\varepsilon/2$,
there exists an $N$-reduced hypergraph $H$
with density at least $\pi-\varepsilon/2$
which does not embed $F$.

Apply Lemma~\ref{lem:lowres} to this hypergraph $H$
to find a subset $U\subseteq[N]$ of $m$ indices,
and for each $\alpha, \beta\in U$
we find a multiset $S_{\alpha,\beta}$ of $s$ vertices in $V_{\alpha, \beta}$,
such that the $m$-reduced hypergraph $H'$
induced by $H$ on the sets $S_{\alpha, \beta}$
has density at least $\pi-\varepsilon$.

For every pair $\alpha, \beta\in U$,
let $v_{\alpha, \beta}^1, v_{\alpha, \beta}^2, \dots, v_{\alpha, \beta}^s$
be an ordering of the vertices of $S_{\alpha, \beta}$.
For every triple $\alpha<\beta<\gamma$ in $U$,
this identifies the edges of the constituent 
$A_{\alpha, \beta, \gamma}$ of $H'$
with a subset of $[s]^3$.
There are $2^{s^3}$ such subsets.
Applying Ramsey's theorem,
there exists a subset $U'\subseteq U$ of size $n$
and a subset $\mathcal{A}\subseteq [s]^3$
such that for all triples $\alpha<\beta<\gamma$ in $U'$
the constituent $A_{\alpha, \beta, \gamma}$
is identified with $\mathcal{A}$.
Moreover, since $H'$ has density at least $\pi-\varepsilon$,
we have that $|\mathcal{A}|\geq(\pi-\varepsilon)s^3$.

Let $\cp$ be the palette with color set $[s]$,
where the family of admissible triples is $\mathcal{A}$.
Because $|\mathcal{A}|\geq (\pi-\varepsilon)s^3$,
$\cp$ has density at least $\pi-\varepsilon$.
The $n$-reduced hypergraph induced by $H'$
on the sets $S_{\alpha, \beta}$ with $\alpha, \beta\in U'$
is precisely $\cp[n]$.
Because $H$ does not embed $F$,
$\cp[n]$ does not embed $F$ either
(the fact that the $S_{\alpha,\beta}$
are multisets from $V_{\alpha, \beta}$
rather than subsets is irrelevant,
because an embedding of $F$ in $H'$
can only use one vertex from each $S_{\alpha, \beta}$).
This means by Lemma~\ref{lem:embedding} 
that $F$ does not admit $\cp$,
so $\pa(F)\geq \pi-\varepsilon$.
By taking $\varepsilon\rightarrow 0$
we conclude that $\pa(F)\geq \pi=\pu(F)$.\end{proof}

\section{Vertex-pair density and degree-dense palettes}\label{sec:pve}

In this section we will prove Theorem~\ref{thm:pve},
which is an analogue of Theorem~\ref{thm:palette}
for the parameter $\pve$.
In order to define $\pve$, we need a stronger notion of uniform density.
Given sets $A\subseteq V(H)$ and $B\in {V(H) \choose 2}$,
we denote \[E_{\pedge}(A,B)=\{a\in A, \{b,c\}\in B: \{a,b,c\}\in E(H)\}.\]

\begin{definition}
A $3$-graph $H$ is said to be \emph{$(d,\varepsilon,\pedge)$-dense}
if for every subset $A\subseteq V(H)$
and every $B\subseteq {V(H)\choose 2}$
we have $E_{\pedge}(A,B)\geq d|A||B|-\varepsilon|V(H)|^3$.
The \emph{vertex-pair Tur\'an density} $\pve(F)$
of a $3$-graph $F$
is defined as the infimum of the values of $d$, 
for which there exists $\varepsilon>0$ and $N$ such that 
every $(d,\varepsilon,\pedge)$-dense hypergraph
on at least $N$ vertices
contains $F$ as a subgraph.
\end{definition}

Reiher, R\"odl and Schacht proved in~\cite{ReiRS16}
that $\pve(K_4^{(3)})=1/2$.
Additional results and bounds concerning $\pve$
can be found in~\cite{Rei20}.

The connection between $\pve$ and palettes
manifests itself not through the density of a palette,
as was the case in $\pu$,
but through its minimum degree.

\begin{definition}
The minimum degree of a palette $\cp(\mathcal{C},\mathcal{A})$, 
denoted by $\delta(\cp)$, 
is the largest value of $d$ such that, for all $a\in\mathcal{C}$,
\[\{(b,c):(a,b,c)\in\mathcal{A}\},\{(b,c):(b,a,c)\in\mathcal{A}\},\{(b,c):(b,c,a)\in\mathcal{A}\}\geq d|\mathcal{C}|^2.\]
The parameter $\pdpal(F)$
 of a $3$-graph $F$ is defined as 
\[\pdpal(F):=\sup\{\delta(\cp):\cp\text{ palette, }F\text{ does not admit }\cp\}.\]
\end{definition}

Given a palette $\cp$, 
we can construct a tripartite $3$-graph $H_{\cp}$ 
on three copies $\mathcal{C}_1,\mathcal{C}_2,\mathcal{C}_3$ of $\mathcal{C}$, 
by taking as edges the ordered triples in $\mathcal{A}$.
Then $d(\cp)$ relates to the edge-density in $H_{\cp}$,
while $\delta(\cp)$ relates to its minimum degree.

We can show that $\pdpal(F)\leq \pve(F)$. 
Indeed, using the same construction mentioned in Section~\ref{sec:palettes},
we can use a palette $\cp$ to generate 
a $(\delta(\cp), o(1), \pedge)$-dense, $F$-free $3$-graph $H_n$.
We answer another question mentioned by Reiher
by showing that equality always holds:

\begin{theorem}\label{thm:pve}
For every $3$-graph $F$, we have $\pve(F)=\pdpal(F)$.
\end{theorem}

In broad terms, the proof of Theorem~\ref{thm:pve}
is similar to that of Theorem~\ref{thm:palette}.
Before sketching the proof,
we need to see how $\pve$ relates
to reduced hypergraphs.
We say that the $N$-reduced hypergraph $H$,
with vertex sets $V_{\alpha, \beta}$ 
and constituents $V_{\alpha, \beta. \gamma}$,
has degree-density at least $d$
if for all triples of indices $\alpha, \beta, \gamma$,
every vertex $v\in V_{\alpha, \beta}$
has degree at least $d|V_{\alpha,\gamma}||V_{\beta, \gamma}|$
in $A_{\alpha, \beta, \gamma}$.
We can now give the analogue of Proposition~\ref{prop:reiher}:

\begin{proposition}[{\cite[Theorem 3.3]{Rei20}}]
\label{prop:reiherpve}
Let $F$ be a $3$-graph. Then $\pve(F)$ 
is the supremum of the values of $d$
such that, for all $N$,
there exists an $N$-reduced hypergraph 
with degree-density at least $d$ which does not embed $F$.
\end{proposition}

Next we will sketch the proof of Theorem~\ref{thm:pve}.
Let $\pi=\pve(F)$ and let $\varepsilon>0$.
The proof starts applying Proposition~\ref{prop:reiherpve}
to find an $N$-reduced hypergraph
with degree-density at least $\pi-\varepsilon/2$.
We apply an analogous of Lemma~\ref{lem:lowres}
to find $U\subseteq[N]$ of size $m$
and multisets $S_{\alpha, \beta}$ of size $s$
such that each constituent $A_{\alpha,\beta,\gamma}$
induced on these multisets
has minimum degree at least $(\pi-\varepsilon)s^2$.
Applying Ramsey's theorem, we find a subset $U'\subseteq U$
of $n$ indices where all constituents are the same.
This corresponds to an $n$-reduced hypergraph of the form $\cp[n]$
for some palette $\cp$ with minimum degree at least $d-\varepsilon$
which is not admitted by $F$.

The key point here is thus
adapting the proof of Lemma~\ref{lem:lowres},
which concerns the density of the reduced hypergraph $H$,
to make it about degree-density instead.
The result is the following lemma:

\begin{lemma}\label{lem:pvelowres}
For all $m$ there exists $s$, such that for all $\varepsilon>0$ there exists $N$ 
such that the following holds: 
if $H$ is an $N$-reduced hypergraph with degree-density at least $d$, 
there exists a subset $U\subseteq [N]$ of $m$ indices, 
and for each $\alpha,\beta$ in $U$ 
there exists a multiset $S_{\alpha,\beta}$ of $s$ vertices in $V_{\alpha,\beta}$,
such that the $m$-reduced hypergraph induced by $H$
on the sets $S_{\alpha,\beta}$ has degree-density at least $d-\varepsilon$.
\end{lemma}

Once again, the proof of Lemma~\ref{lem:pvelowres}
is similar to that of Lemma~\ref{lem:lowres}.
We follow an algorithmic approach,
with $r$ rounds, where on each round
we select $t$ vertices from each $V_{\alpha, \beta}$.
We want to ensure that
for all $\alpha, \beta, \gamma\in U_r$
and every $1\leq i<j<k\leq r$,
the tripartite graph between
$T_{\alpha, \beta}^{(i)}$, $T_{\alpha, \gamma}^{(j)}$
and $T_{\beta, \gamma}^{(k)}$
has minimum degree at least $(\pi-\varepsilon/2)s^2$.
Adding up all triples $i,j,k$
and all permutations of $\alpha, \beta, \gamma$,
the resulting constituent $A_{\alpha, \beta, \gamma}$
has minimum degree at least $(\pi-\varepsilon/2)(r-1)(r-2)t^2$
which is at least $(\pi-\varepsilon)s^2$
for $r=\lceil 6\varepsilon^{-1}\rceil$.

The main difference comes in the choice of invariants.
It would be natural to require
that the tripartite graphs on
$T_{\alpha, \beta}^{(i)}$, $V_{\alpha, \gamma}$, $V_{\beta, \gamma}$
and on $T_{\alpha, \beta}^{(i)}$, $T_{\alpha, \gamma}^{(j)}$, $V_{\beta, \gamma}$
also satisfy some minimum degree condition.
However, this is not attainable.
For example, if all of the sets $V_{\alpha, \beta}$,
$V_{\alpha, \gamma}$ and $V_{\beta, \gamma}$ are large and of equal size,
take the complete tripartite $2$-graph $G$ between them,
and randomly color the edges red and blue.
Take $H$ to be the $3$-graph whose edges
are the blue triangles of $G$.
With high probability, for all choices of the $t$-set $T_{\alpha, \beta}^{(i)}$
there will be vertices $v\in V_{\alpha, \gamma}$
which are connected to all vertices in $T_{\alpha, \beta}^{(i)}$
through red edges.
As such, the tripartite graph on
$T_{\alpha, \beta}^{(i)}$, $V_{\alpha, \gamma}$, $V_{\beta, \gamma}$
has minimum degree 0.

To get around constructions like these,
we will modify our invariants
to allow for a small number of exceptions
within the sets $V_{\alpha, \beta}$.
We denote by $\mu(A,B,C)$
the edge-density within $H$
of the tripartite graph between $A$, $B$ and $C$.

\begin{enumerate}[label=(\roman*)]
\item\label{item:iti} $\mu(v, T_{\alpha, \gamma}^{(j)}, V_{\beta, \gamma})\geq \pi-\varepsilon/4$ for all distinct $i,j\in[r]$, all $\alpha, \beta, \gamma\in U_{\max\{i,j\}}$ and all $v\in T_{\alpha,\beta}^{(i)}$.
\item\label{item:itii} $\mu(v, T_{\alpha, \gamma}^{(j)}, T_{\beta, \gamma}^{(k)})\geq \pi-\varepsilon/2$ for all distinct $i,j,k\in[r]$, all $\alpha, \beta, \gamma\in U_{\max\{i,j,k\}}$ and all $v\in T_{\alpha, \beta}^{(i)}$.
\item\label{item:itiii} $\mu(v, T_{\alpha, \gamma}^{(i)}, V_{\beta, \gamma})\geq \pi-\varepsilon/4$ for all $i\in[r]$, all $\alpha, \beta, \gamma\in U_i$ and at least $\left(1-\frac{1}{(1000rt)^3}\right)|V_{\alpha, \beta}|$ vertices $v\in V_{\alpha,\beta}$.
\item\label{item:itiv} $\mu(v, T_{\alpha, \gamma}^{(j)}, T_{\beta, \gamma}^{(i)})\geq \pi-\varepsilon/2$ for all distinct $i,j\in[r]$, all $\alpha, \beta, \gamma\in U_{\max\{i,j\}}$ and at least $\left(1-\frac{2}{(1000rt)^3}\right)|V_{\alpha, \beta}|$ vertices $v\in V_{\alpha,\beta}$.
\end{enumerate}

The proof is now similar to
the proof of Lemma~\ref{lem:lowres}.
Suppose that the four invariants
hold after the $k-1$-th round,
and we now need to select the sets $T_{\alpha, \beta}^{(k)}$.
We set $B_{\alpha, \beta}^\gamma\subseteq(V_{\alpha, \beta})^t$
to be the set of choices of $T_{\alpha, \beta}^{(k)}$
for which one of the invariants would be broken
for some permutation of $\alpha, \beta, \gamma$.
We show that $|B_{\alpha, \beta}^\gamma|\leq 0.1|V_{\alpha, \beta}|^t$,
which allows us to use Lemma~\ref{lem:selection}.

The number of choices of $T_{\alpha, \beta}^{(i)}$
that break each invariant is at most $0.01|V_{\alpha, \beta}|^t$.
We will sketch here the count for~\ref{item:iti} and \ref{item:itiv}.
The other two invariants are similar.
We take a value of $t$ large enough that $e^{-(\varepsilon/4)^2t}<1/(1000rt)^4$.

There are four ways in which invariant~\ref{item:iti}
can break on $\alpha, \beta, \gamma$
when selecting $T_{\alpha, \beta}^{(k)}$.
For some $j<k$, one can have that
$\mu(v,T_{\alpha, \gamma}^{(j)},V_{\beta, \gamma})<\pi-\varepsilon/4$
for some $v\in T_{\alpha, \beta}^{(k)}$,
or $\mu(v, T_{\alpha, \beta}^{(k)}, V_{\beta, \gamma})<\pi-\varepsilon/4$
for some $v\in T_{\alpha, \gamma}^{(j)}$,
or the same two scenarios swapping $\alpha$ and $\beta$.
Because invariant~\ref{item:itiii} holds before the $k$-th round,
the first scenario can only happen if $v$ is selected
from a subset of at most $|V_{\alpha, \beta}|/(1000rt)^3$ vertices.
The probability that at least one of the $t$ vertices
selected on the $k$-th round comes from this set
is at most $1/1000$.
On the other hand, because $\mu(v, V_{\alpha, \beta}, V_{\beta, \gamma})\geq\pi$
holds for all $v\in V_{\alpha, \gamma}$,
for each individual $v\in T_{\alpha, \gamma}^{(j)}$
the second scenario holds
with probability at most $e^{-(\varepsilon/4)^2t}$
by Lemma~\ref{lem:hoeffding}.
The third and fourth scenarios are analogous to the previous two.
Adding these up, for all choices of $j$ and $v$,
the total probability that $T_{\alpha, \beta}$ breaks invariant~\ref{item:iti}
 is less than $0.01$.
 
 There are two ways in which invariant~\ref{item:itiv}
 can break on $\alpha, \beta, \gamma$
when selecting $T_{\alpha, \beta}^{(k)}$.
For some $j<k$, one can have that
$\mu(v, T_{\alpha, \beta}^{(k)}, T_{\beta, \gamma}^{(j)})<\pi-\varepsilon/2$
holds for more than $2|V_{\alpha,\gamma}|/(1000rt)^3$ vertices $v\in V_{\alpha, \gamma}$,
or the same scenario swapping $\alpha$ and $\beta$.
Let $L$ be the set of vertices $v\in V_{\alpha, \gamma}$
with $\mu(v, V_{\alpha, \beta}, T_{\beta, \gamma}^{(j)})<\pi-\varepsilon/4$.
Since invariant~\ref{item:itiii} holds before the $k$-th round,
we have that $|L|\leq|V_{\alpha, \gamma}|/(1000rt)^3$.
For each $v\notin V_{\alpha, \beta}$, by Lemma~\ref{lem:hoeffding},
$\Pr\left(\mu(v, T_{\alpha, \beta}^{(k)}, T_{\beta, \gamma}^{(j)})<\pi-\varepsilon/2\right)\leq e^{-(\varepsilon/4)^2t}\leq 1/(1000rt)^4$.
By Markov's inequality,
the probability that this happens for more than $|V_{\alpha, \gamma}|/(1000rt)^3$
vertices $v\notin L$ is at most $1/1000rt$.
Adding up over all choices of $j$ and the symmetric case
swapping $\alpha$ and $\beta$,
the probability that the choice of $T_{\alpha, \beta}^{(k)}$
breaks invariant~\ref{item:itiv}
is at most $0.01$.

Proceeding the same way with invariants~\ref{item:itii} and \ref{item:itiii},
we conclude that $|B_{\alpha, \beta}^\gamma|\leq 0.1|V_{\alpha, \beta}|^t$.

\section{Applications}\label{sec:applications}

\subsection{The broken tetrahedron}

In this section we will use Theorem~\ref{thm:palette}
to give a short proof of $\pu(K_4^{(3)-})=1/4$.
This was first proved by Glebov, Kr\'al' and Volec~\cite{GleKV16}
using the flag algebra method,
and independently by Reiher, R\"odl and Schacht~\cite{ReiRS18a}
using the hypergraph regularity method.
While our proof of Theorem~\ref{thm:palette}
implicitly uses regularity in the form of Proposition~\ref{prop:reiher},
once Theorem~\ref{thm:palette} is treated as a black box
neither regularity nor flag algebras are required.

\begin{theorem}[\cite{GleKV16, ReiRS18a}]\label{thm:k4-}
$\pu(K_4^{(3)-})=1/4$.
\end{theorem}

\begin{proof}
The hypergraph $K_4^{(3)-}$ does not admit the palette $\mathcal{Q}$ 
with color set $\{1,2\}$ and triples $\{(1,2,1),(2,1,2)\}$,
so $\pa(K_4^{(3)-})\geq d(\mathcal{Q})=1/4$.

Let $\cp=(\mathcal{C}, \mathcal{A})$ be a palette 
that $K_4^{(3)-}$ does not admit.
Construct two auxiliary directed graphs $G_L$ and $G_R$
on the vertex set $\mathcal{C}$.
Given $a,b\in \mathcal{C}$, not necessarily distinct,
we add the directed edge $\vec{ab}$ in $G_L$
if there exists $c\in \mathcal{C}$ such that $(a,b,c)\in\mathcal{A}$.
We add $\vec{ab}$ in $G_R$ if there exists $c\in \mathcal{C}$
such that $(c,b,a)\in \mathcal{A}$.

We claim that $G_L$ does not contain three edges
of the form $\vec{ab},\vec{ac},\vec{bc}$,
with $u,v,w$ not necessarily distinct.
Indeed, given four vertices with the order $v_1\prec v_2\prec v_3\prec v_4$,
by coloring the pair $v_1v_2$ in color $a$,
$v_1v_3$ in color $b$ and $v_1v_4$ in color $c$,
with the right choices of colors for $v_2v_3$, $v_2v_4$ and $v_3v_4$
we have that $K_4^{(3)-}$ admits $\cp$.
The same argument applies to $G_R$.

%The in-neighborhood $N^-_L(c)$ of each color 
%$c\in\mathcal{C}$ in $G_L$ forms an independent set.
%Let $d^-_L(c)=|N^-_L(c)|$.
%Let $a\in\mathcal{C}$ be the color with the largest value of $d^-_L$,
%and let $\Delta=d^-_L(a)$.
%By definition, $d^-_L(c)\leq \Delta$ for all $c\notin N_L^-(a)$.
%Since $N_L^-(a)$ is independent, $d^-_L(c)\leq |\mathcal{C}|-\Delta$ for all $c\in N^-_L(c)$.
%Adding up over all $c$, we get
%\[\sum\limits_{c\in\mathcal{C}}d_L^-(c)^2\leq (|\mathcal{C}|-\Delta)\Delta^2+\Delta(|\mathcal{C}|-\Delta)^2\leq \frac{|\mathcal{C}|^3}{4}.\]

Let $d^+_L(c)$ and $d^-_L(c)$ denote the number
of out-neighbors and in-neighbors of $c$ in $G_L$.
Let $S$ be the set of triples $(a,b,c)$
with $\vec{ac},\vec{bc}\in G_L$.
Note that, for each such triple, we have $\vec{ab}\notin G_L$.
Therefore, by double-counting,

\[\sum\limits_{c\in\mathcal{C}}d^-_L(c)^2=|S|\leq \sum\limits_{a\in\mathcal{C}}(|\mathcal{C}|-d^+_L(a))d^+_L(a)\leq \sum\limits_{a\in\mathcal{C}}\frac{|\mathcal{C}|^2}{4}=\frac{|\mathcal{C}|^3}{4}.\]

The same happens in $G_R$.
In each triple $(a,b,c)\in \mathcal{A}$,
we have $\vec{ab}\in G_L$ and $\vec{cb}\in G_R$,
which means that
\[|\mathcal{A}|\leq \sum\limits_{b\in\mathcal{C}}d^-_L(b)d^-_R(b)\leq \sum\limits_{b\in\mathcal{C}}\frac{d^-_L(b)^2+d^-_R(b)^2}{2}\leq \frac{|\mathcal{C}|^3}{4},\]
and so $d(\cp)\leq 1/4$. We conclude that $\pa(K_4^{(3)-})=1/4$, and by Theorem~\ref{thm:palette}, $\pu(K_4^{(3)-})=1/4$.
\end{proof}

\subsection{Infinitely many values for uniform Tur\'an density}

As another application of Theorem~\ref{thm:palette}
we will prove Theorem~\ref{thm:inffamily},
showing that for every $k\geq 2$
there exists a $3$-graph $F_k$
with $\pu(F_k)=\frac12-\frac{1}{2k}$.
Consider the palette $\cp_k=([k], \mathcal{A}_k)$
where $\mathcal{A}_k=\{(x,y,z)\in [k]^3:x<z\}$.
One can easily check that $d(\cp_k)=\frac12-\frac{1}{2k}$.
The following claim will play an important role in the proof:

\begin{claim}\label{claim:path}
Let $\cp=(\mathcal{C}, \mathcal{A})$ be a palette
with $d(\cp)>\frac12-\frac{1}{2k}$.
Then $\mathcal{A}$ contains $k$ triples
of the form $(a_1, b_1, a_2),(a_2, b_2, a_3),\dots, (a_k, b_k, a_{k+1})$,
where the colors $a_i$ and $b_j$
are not necessarily distinct.
\end{claim}

\begin{proof}
Consider a directed graph $G$ on $\mathcal{C}$,
where we take a directed edge $\vec{uv}$
between two colors
if there exists a third color $w\in\mathcal{C}$
with $(u,w,v)\in\mathcal{A}$.
Our statement is equivalent to showing
that $G$ contains a walk on $k$ edges.

Label each vertex $v$ by the length of
the longest walk ending at $v$.
If for any $v$ this length is unbounded
or it is at least $k$, we are done.
Let $\mathcal{C}_i$ be the set of colors
from $\mathcal{C}$ receiving the label $i$.
Each triple in $(u,w,v)\in\mathcal{A}$
has $u$ in a lower set than $w$.
Thus 
\[|\mathcal{A}|\leq|\mathcal{C}|\sum\limits_{0\leq i<j\leq k-1}|\mathcal{C}_i||\mathcal{C}_j|=\frac12|\mathcal{C}|\left(|\mathcal{C}|^2-\sum\limits_{i=0}^{k-1}|\mathcal{C}_i|^2\right)\leq\frac12\left(1-\frac1k\right)|\mathcal{C}|^3,\]
meaning that $d(\cp)\leq\frac12-\frac{1}{2k}$.\end{proof}

Next we will construct a hypergraph $H$
which will be used as an intermediate step
in our construction of $F_k$.
A hypergraph is said to be \emph{linear}
if every pair of edges intersects in at most one vertex.

\begin{lemma}~\label{lem:monotone}
For every $k\geq 3$
there exists a positive integer $n$
and a linear $k$-graph $H$
on the vertex set $[n]$
such that for every permutation $\sigma$ of $[n]$
there exists an edge $e$ in $H$
such that $\sigma$ is monotone on the vertices of $e$.
\end{lemma}

\begin{proof}
Consider the random hypergraph $H'=H^{(k)}(n,p)$, 
where the vertex set is $[n]$ for a very large integer $n$
and every $k$-tuple of vertices becomes an edge
independently with probability $p$.
Set $p=n^{3/2-k}$.
Obtain $H$ from $H'$ by removing every edge
that intersects another edge from $H'$
in at least two vertices.
The resulting $H$ is therefore linear.

We claim that $E(H'\setminus H)\leq n^{5/4}$ 
with probability tending to 1.
Indeed, we can describe each pair of edges $e_1, e_2$
intersecting in two or more vertices
as a sequence of $2k-2$ vertices
$(u_1, u_2, v_1, v_2, \dots, v_{k-2}, w_1, w_2, \dots, w_{k-2})$,
where $e_1=u_1u_2v_1v_2\dots v_{k-2}$ and $e_2=u_1u_2w_1w_2\dots w_{k-2}$.

There are at $n^{2k-2}$ sequences of $2k-2$ vertices,
and for each of them, the probability that it corresponds
to two edges in $H'$ is at most $n^{3-2k}$.
Therefore, the expected number of pairs of edges
intersecting in two or more vertices is $O(n)$,
and by Markov's inequality, the probability
that $E(H'\setminus H)\geq n^{5/4}$ is $O(n^{-1/4})=o(1)$.

Next we will show that, with high probability,
for every permutation $\sigma$ there exists
an edge in $H$ for which $\sigma$ is monotone.
In fact, we claim that if $E(H'\setminus H)\leq n^{5/4}$,
for any fixed permutation $\sigma$
the probability that $\sigma$ is not monotone on any edge of $H$
is $o(1/n!)$, from which our claim follows by the union bound.

By the Erd\H{o}s-Szekeres theorem, 
among any $k^2$ elements of $n$
there exists a $k$-tuple on which $\sigma$ is monotone.
Therefore the number of $k$-tuples of $[n]$ 
on which $\sigma$ is monotone is at least
${n \choose k}/{k^2 \choose k}=\Omega(n^k)$.
The probability that 
fewer than $n^{5/4}$ of these $k$-tuples are edges of $H'$
is at most \[{n^k \choose n^{5/4}}(1-p)^{\Omega(n^k)}= n^{O(n^{5/4})}e^{-\Omega(pn^k)}=e^{-\Omega(n^{3/2})}=o(1/n!),\] 
as we wanted to show.\end{proof}

\begin{proof}[Proof of Theorem~\ref{thm:inffamily}]
Let $H$ be a linear $k+2$-graph on $[n]$
as in Lemma~\ref{lem:monotone}.
For every edge $e$ with vertices
$v_1<v_2<\dots< v_{k+2}$,
place an edge on the triple $v_iv_{i+1}v_{i+2}$ for each $i\in[k]$
to create the $3$-graph $F_k$.

On the one hand, $F_k$ admits every palette $\cp$
with density greater than $\frac12-\frac{1}{2k}$.
This is because, by Claim~\ref{claim:path},
there are $k$ admissible triples
of the form $(a_i,b_i,a_{i+1})$ for $i\in[k]$.
Color each pair of the form $v_iv_{i+1}$ with the color $a_i$,
and each pair $v_iv_{i+2}$ with the color $b_i$.
Because $H$ is linear, this coloring is consistent,
i.e., no pair of vertices receives more than one color.
Because each edge is of the form $v_iv_{i+1}v_{i+2}$,
the colors of $v_iv_{i+1}$, $v_iv_{i+2}$ and $v_{i+1}v_{i+2}$
form an admissible triple in $\cp$, hence $F_k$ admits $\cp$.

On the other hand, we will show
that $F_k$ does not admit $\cp_k$.
Assume for contradiction that $F_k$ admits $\cp_k$,
with the vertex order $\preceq$ and the coloring function $\varphi$.
There must exist an edge $e\in E(H)$ 
for which the order $\preceq$ is monotone 
with respect to the natural order $\leq$ on $[n]$.
Let $v_1<v_2<\dots<v_{k+2}$ be the vertices of $e$.
If $\preceq$ is increasing on $e$,
then $\varphi(v_1v_2)<\varphi(v_2v_3)<\dots<\varphi(v_{k+1}v_{k+2})$,
which is impossible since $\cp_k$ only has $k$ colors.
The same happens if $\preceq$ is decreasing on $e$,
producing a contradiction.
We conclude that $\pa(F_k)=\frac{1}{2}-\frac{1}{2k}$,
which by Theorem~\ref{thm:palette} is the same value as $\pu(F_k)$.\end{proof}

\subsection{Vertex-pair Tur\'an density 0}\label{sec:pve0}

As an application of Theorem~\ref{thm:pve},
we will give two characterizations of the 
family of $3$-graphs $F$ with $\pve(F)=0$.
It is unknown whether either of these
matches the description
obtained by Reiher, R\"odl and Schacht.

Let $\mathcal{A}\subseteq \NN^3$ be a set
of triples of positive integers.
We say that $\mathcal{A}$ is \emph{min-degenerate} if
the three entries of each element of $\mathcal{A}$ are different,
and for each pair $(a_1, a_2, a_3)\neq(b_1, b_2, b_3)\in\mathcal{A}$,
if $t\in \{a_1, a_2, a_3\}\cap\{b_1, b_2, b_3\}$ then
at least one of $t=\min\{a_1, a_2, a_3\}$ or $t=\min\{b_1, b_2, b_3\}$ holds.
We say that a $3$-graph $F$ is \emph{min-layered} if
there exists an ordering $\preceq$ of $V(F)$
and a function $\varphi:{V(F) \choose 2}\rightarrow\NN$
such that the set \[\mathcal{A}=\{(\varphi(uv),\varphi(uw),\varphi(vw)):uvw\in E(F),u\prec v\prec w\}\]
is min-degenerate.

Let $F_{\{a,b\}}$ be the free group on generators $a,b$.
The elements of this group are
the words of finite length on the alphabet $\{a,b,a^{-1},b^{-1}\}$,
where no letter is adjacent to its inverse.
The product of two words is equal to their concatenation,
after cancelling any adjacent inverse pairs.

The free group satisfies the universal property:
given any group $G$ and elements $x,y\in G$,
there exists a unique homomorphism 
$f:F_{\{a,b\}}\rightarrow G$
with $f(a)=x$ and $f(b)=y$.
For every $k$ there exists a finite group $G$
and elements $x,y\in G$
such the function $f$ 
restricted to $B_k$ is injective,
where $B_k$ is the set of all elements in $F_{\{a,b\}}$
where the length of the corresponding word is at most $k$ 
(see for example~\cite{DixPSS03}).

\begin{theorem}\label{thm:pve0}
Let $F$ be a $3$-graph. The following are equivalent:
\begin{enumerate}[label=(\roman*)]
\item\label{item:pve0} $\pve(F)=0$.
\item\label{item:fgroup} There exists an ordering $\preceq$ of $V(F)$, 
and a function $\psi:{V(F) \choose 2}\rightarrow F_{\{a,b\}}$,
such that for every $uvw\in E(F)$ with $u\prec v\prec w$
we have $(\psi(uv),\psi(uw),\psi(vw))=(x,xa,xb)$
for some $x\in F_{\{a,b\}}$.
\item\label{item:minlayer} $F$ is min-layered.
\end{enumerate}
\end{theorem}

\begin{proof}

\emph{$\ref{item:pve0}\Rightarrow\ref{item:fgroup}$.}
Let $\hat F$
be the graph on the vertex set ${V(F) \choose 2}$
where two pairs of the form $uv,uw$
are connected through an edge
if $uvw$ is an edge in $F$.
Let $d$ be the maximum distance
between two vertices of the same component of $\hat F$.
Let $G$ be a finite group, and $x,y\in G$,
such that the homomorphism $f$
with $f(a)=x$ and $f(b)=y$ 
is injective in $B_{2d+1}$.

Construct a palette $\mathcal{P}=(G,\mathcal{A})$
where $\mathcal{A}$ is the family of triples
of the form $(z,zx,zy)$, with $z\in G$.
This palette has positive minimum degree,
so since $\pdpal(F)=0$ by Theorem~\ref{thm:pve},
$F$ admits $\cp$, with some vertex order $\preceq$
and some function $\varphi:{V(F) \choose 2}\rightarrow G$.

Let $S$ be a set containing one vertex
from each component of $\hat F$.
If we left-multiply the image $\varphi(uv)$
of all vertices $uv$ in some component of $\hat F$
by the same element of $G$,
then the image of each triple $(uv,uw,vw)$
is still in $\mathcal{A}$ for each edge $uvw$ of $F$
with $u\prec v\prec w$.
Therefore, we can assume that $\varphi(uv)=0$
for all $uv\in S$.
Then the image under $\varphi$ of each vertex $uv\in V(\hat F)$
is the image under $f$ of some element of $B_{2d}$,
a fact that can be checked by taking a path 
from $uv$ to an element of $S$.
Because $f$ is injective on $B_{2d+1}$,
we can define $\psi(uv)=f^{-1}(\varphi(uv))$.

For every edge $uvw\in E(F)$ with $u\prec v\prec w$,
we have the relation $f(\psi(uw))=\varphi(uw)=\varphi(uv)\cdot x=f(\psi(uv))\cdot x=f(\psi(uv)\cdot a)$.
Since $\psi(uv)$ is in the image of $B_{2d}$,
$\psi(uv)\cdot a$ is in the image of $B_{2d+1}$.
But $f$ is injective on $B_{2d+1}$,
so $\psi(uw)=\psi(uv)\cdot a$. 
Similarly, $\psi(vw)=\psi(uv)\cdot b$.
We conclude that $F$ satisfies \ref{item:fgroup}. 

\emph{$\ref{item:fgroup}\Rightarrow\ref{item:minlayer}$.}
Suppose that $\ref{item:fgroup}$ holds.
Take an ordering $\prec$ of the elements in $F_{\{a,b\}}$
in the image of $\psi$,
where for any $x,y\in F_{\{a,b\}}$,
if the word $x$ is shorter than the word $y$ then $x<y$.
Words of the same length are ordered arbitrarily.
Take a function $f$ from the image of $\psi$ to $\NN$ preserving this order.
We claim that the set $\mathcal{A}$ containing
the images of the triples $(x,xa,xb)$ is min-degenerate.
Hence, $F$ is min-layered.

It is clear that the three elements $(x,xa,xb)$ are always distinct.
For every $x\in F_{\{a,b\}}$,
there are exactly three triples containing $x$,
namely $(x,xa,xb)$, $(xa^{-1},x,xa^{-1}b)$ and $(xb^{-1},xb^{-1}a,x)$.
If $x$ is non-empty, there is exactly one of those triples
in which $x$ is not the minimal element.
It is the first triple if the last letter of $x$ is $a^{-1}$ or $b^{-1}$,
the second one if the last letter is $a$, 
and the third one if the last letter is $b$.
This shows that $\mathcal{A}$ is min-degenerate.

\emph{$\ref{item:minlayer}\Rightarrow\ref{item:pve0}$.}
Suppose that $F$ is min-layered,
as shown by the ordering $\preceq$, 
the function $\varphi$ and the min-degenerate set $\mathcal{A}$.
Let $\cp=(\mathcal{C},\mathcal{A}')$ be any palette with $\delta(\cp)>0$.
We will construct a function $\psi:\NN\rightarrow\mathcal{C}$
satisfying that the image of each triple of $\mathcal{A}$ is in $\mathcal{A}'$.
This shows that $F$ admits $\cp$, so $\pdpal(F)=0$
and, by Theorem~\ref{thm:pve}, $\pve(F)=0$.

We go through the triples $(x,y,z)\in\mathcal{A}$
in increasing order of $\min\{x,y,z\}$,
and in each step we define the values
of $\psi(x),\psi(y),\psi(z)$ which have not been defined yet.
Assume, w.l.o.g., that on one of these steps
we have $x<y<z$.
By the min-degenerate structure of $\mathcal{A}$,
we know that $y$ and $z$ have not appeared
in any previously considered triple,
and hence their images are undefined at this stage.
If $\psi(x)$ is undefined, choose its value arbitrarily.
Let $\psi(x)=r$. Since $\delta(\cp)>0$,
there exists a triple in $\mathcal{A}'$
of the form $(r,s,t)$.
Set $\psi(y)=s$ and $\psi(z)=t$.
Continuing with this procedure,
the image of every triple of $\mathcal{A}$ is in $\mathcal{A}'$.

\end{proof}

\section{Cherry Tur\'an density}\label{sec:remarks}

In addition to $\pu$ and $\pve$, 
there is a third variant of uniform Tur\'an density
for which Reiher asked about its relation to palettes.
Let $H$ be a $3$-graph.
Given two sets $A,B\subseteq V(H)^2$,
we denote
\[K_{\pcherry}(A,B)=\{(a,b,c)\in V(H)^3: (a,b)\in A, (a,c)\in B\}.\]
\[E_{\pcherry}(A,B)=|\{(a,b,c)\in K_{\pcherry}(A,B):abc\in E(H)\}|.\]

\begin{definition}
A $3$-graph $H$ is said to be \emph{$(d,\varepsilon,\pcherry)$-dense}
if for every $A,B\subseteq V(H)^2$
we have $E_{\pcherry}(A,B)\geq d|K_{\pcherry}(A,B)|-\varepsilon|V(H)|^3$.
The \emph{cherry Tur\'an density} $\pch(F)$
of a $3$-graph $F$
is defined as the infimum of the values of $d$, 
for which there exists $\varepsilon>0$ and $N$ such that 
every $(d,\varepsilon,\pcherry)$-dense hypergraph
on at least $N$ vertices
contains $F$ as a subgraph.
\end{definition}

The cherry Tur\'an density of a $3$-graph $F$, 
by analogy to Theorem~\ref{thm:palette}
and Theorem~\ref{thm:pve},
would be related to the minimum codegree of palettes.

\begin{definition}
The minimum codegree of a palette $\cp(\mathcal{C},\mathcal{A})$, 
denoted by $\delta_{\operatorname{cd}}(\cp)$, 
is the largest value of $d$ such that, for all $a,b\in\mathcal{C}$,
\[\{c:(a,b,c)\in\mathcal{A}\},\{c:(a,c,b)\in\mathcal{A}\},\{c:(c,a,b)\in\mathcal{A}\}\geq d|\mathcal{C}|.\]

The parameter $\pdch(F)$
 of a $3$-graph $F$ is defined as 
\[\pdch(F):=\sup\{\delta_{\operatorname{cd}}(\cp):\cp\text{ palette, }F\text{ does not admit }\cp\}.\]
\end{definition}

Reiher asked whether $\pch(F)=\pdch(F)$ for all $3$-graphs $F$.
By a result analogous to Proposition~\ref{prop:reiher}
and Proposition~\ref{prop:reiherpve},
$\pch$ is related to the minimum codegree
of the constituents of a reduced hypergraph.
Therefore, to answer Reiher's question,
it would be enough to prove an analogous of
Lemma~\ref{lem:lowres} about codegree.

There is one important reason why our method
does not easily generalize to take codegree into account.
In the algorithm used in the proof of Lemma~\ref{lem:lowres},
we could potentially control the codegree of pairs of vertices
selected in different rounds,
but we have no simple way of
ensuring that pairs of vertices selected on the same round
have large codegree.

Even if we could somehow prove that $\pch(F)=\pdch(F)$,
it is not clear a priori how to classify
the $3$-graphs $F$ with $\pch(F)=0$.
A reasonable analogy to the min-layered $3$-graphs
from Section~\ref{sec:pve0}
is the following definition
for \emph{max-layered} $3$-graph.

Let $\mathcal{A}\subseteq \NN^3$ be a set
of triples of positive integers.
We say that $\mathcal{A}$ is \emph{max-degenerate} if
each element of $\mathcal{A}$ has a unique maximal entry,
and for each pair $(a_1, a_2, a_3)\neq(b_1, b_2, b_3)\in\mathcal{A}$
we have $\max\{a_1, a_2, a_3\}\neq\max\{b_1, b_2, b_3\}$.
We say that a $3$-graph $F$ is \emph{max-layered} if
there exists an ordering $\preceq$ of $V(F)$
and a function $\varphi:{V(F) \choose 2}\rightarrow\NN$
such that the set \[\mathcal{A}=\{(\varphi(uv),\varphi(uw),\varphi(vw)):uvw\in E(F),u\prec v\prec w\}\]
is max-degenerate.

The motivation for this definition
is that max-layered graphs
have $\pdch(F)=0$.
This can be shown
using a greedy algorithm, similar to the step 
$\ref{item:minlayer}\Rightarrow\ref{item:pve0}$
in the proof of Theorem~\ref{thm:pve0}.
Compare this with the notion of layered $3$-graph
from~\cite{DLLWY24},
which is similarly motivated
by an iterative approach.
This connection is not merely a superficial resemblance,
as both layer structures relate
to constructions in $3$-graphs with minimum codegree conditions.

However, showing that $\pdch(F)=0$
implies that $F$ is max-layered
seems to be a harder problem
than proving that $\pdpal(F)=0$
implies that $F$ is min-layered.

\begin{question}
Is it true that $\pdch(F)=0$ if and only if $F$ is max-layered?
\end{question}

\section{Concluding remarks}\label{sec:finremarks}

Since the uniform Tur\'an density
of a $3$-graph $F$ depends only
on the palettes that $F$ admits,
one would like to understand in which situations
will there exist a $3$-graph $F$
admitting all palettes in some family $\{\cp_1, \cp_2, \dots, \cp_m\}$,
but none of $\{\mathcal{Q}_1, \mathcal{Q}_2, \dots, \mathcal{Q}_n\}$.

One particular reason why someone might be interested
in this type of questions
is that it would allow us to prove results
about the set of values of $\pu$
without the need to look at $3$-graphs at all.
Say that, for some $\alpha\in[0,1]$,
we want to show the existence of a $3$-graph with $\pu(F)=\alpha$.
First, we find a palette $\mathcal{Q}$ with $d(\mathcal{Q})=\alpha$.
Second, we find a family of palettes $\{\cp_1, \cp_2, \dots, \cp_n\}$
such that every palette with density strictly greater than $\alpha$
``contains'' some palette $\cp_i$ (whatever that means).
Then, using some black box,
we show the existence of a $3$-graph $F$ which admits all $\cp_i$, but not $\mathcal{Q}$.
This $F$ satisfies $\pu(F)=\alpha$.

There is a natural way to define such a containment relation.
We say that $\cp=(\mathcal{C}_{\cp}, \mathcal{A}_{\cp})$ 
is a \emph{subpalette} of $\mathcal{Q}=(\mathcal{C}_{\mathcal{Q}}, \mathcal{A}_{\mathcal{Q}})$,
and denote it as $\cp\subseteq\mathcal{Q}$,
if there exists a function $f:\mathcal{C}_{\cp}\rightarrow\mathcal{C}_{\mathcal{Q}}$
such that for all $(x,y,z)\in\mathcal{A}_{\cp}$,
we have $(f(x),f(y),f(z))\in\mathcal{A}_{\mathcal{Q}}$.
An important property of this relation is that,
if $\cp\subseteq \mathcal{Q}$,
then every $3$-graph $F$ admitting $\cp$
also admits $\mathcal{Q}$.

The analysis of when there exists a $3$-graph $F$
admitting certain patterns $\{\cp_1, \cp_2, \dots, \cp_m\}$,
but none of $\{\mathcal{Q}_1, \mathcal{Q}_2, \dots, \mathcal{Q}_n\}$
will be the topic of an upcoming paper~\cite{KraKLT24}.
In this paper, a characterization of such families $\{\cp_i\}$ and $\{\mathcal{Q}_j\}$
will be presented.
However, the general characterization
requires the introduction and explanation of several concepts.
For now, as a preview,
we will state the characterization for 
one-on-one comparisons between palettes.

Given a palette $\cp=(\mathcal{C},\mathcal{A})$,
we define its \emph{reverse},
denoted by $\operatorname{rev}(\cp)$,
as the palette $(\mathcal{C},\operatorname{rev}(\mathcal{A}))$, with
\[\operatorname{rev}(\mathcal{A})=\{(z,y,x):(x,y,z)\in\mathcal{A}\}.\]
This palette satisfies that $F$ admits $\operatorname{rev}(\cp)$
if and only if it admits $\cp$,
by reversing the order of the vertices of $F$.
Thus, if $\cp\subseteq \operatorname{rev}(\mathcal{Q})$,
then every $3$-graph which admits $\cp$ also admits $\mathcal{Q}$.

\begin{theorem}\label{thm:oneonone}
Let $\cp$ and $\mathcal{Q}$ be palettes.
Then there exists a $3$-graph $F$
which admits $\cp$ but not $\mathcal{Q}$
iff both $\cp\not\subseteq \mathcal{Q}$
and $\cp\not\subseteq\operatorname{rev}(\mathcal{Q})$ hold.
\end{theorem}

This theorem could have been used
in the proof of Theorem~\ref{thm:inffamily}.
Here, we take the palette $\mathcal{Q}=([k],\mathcal{A})$
with $\mathcal{A}=\{(x,y,z):x<z\}$,
and the palette $\cp=(\mathcal{C},\mathcal{A}')$
with \[\mathcal{C}=\{a_1, a_2, \dots, a_{k+1},b_1, b_2, \dots, b_k\}\quad
\text{and} \quad\mathcal{A}'=\{(a_i, b_i, a_{i+1}):i\in[k]\}.\]
In our proof of Theorem~\ref{thm:inffamily},
we constructed $F$ which admits $\cp$
but not $\mathcal{Q}$ manually,
rather than relying on Theorem~\ref{thm:oneonone}.

\section*{Acknowledgements}

The author would like to thank Dan Kr\'al' and Frederik Garbe
for helpful discussions on the topic.
We would also like to thank Dan Kr\'al', Hong Liu, Laihao Ding and Zhuo Wu
for carefully reading a draft of this paper.

\bibliographystyle{abbrv}
\bibliography{bibliog.bib}
\end{document}